\documentclass[11pt, a4paper,reqno]{amsart}

\usepackage[UKenglish]{babel}
\usepackage[UKenglish]{isodate}
\cleanlookdateon

\usepackage{enumerate}
\usepackage[PS]{diagrams}
\usepackage{xcolor}
\usepackage{pgf}

\newcommand{\dd}[1]{\ensuremath{\mathcal{#1}}}

\newcommand{\iso}{\cong}

\newcommand{\ie}{\textit{i.e.}}

\numberwithin{equation}{section}

\newtheorem{theorem}[equation]{Theorem}

\newtheorem{proposition}[equation]{Proposition}
\newtheorem{lemma}[equation]{Lemma}
\newtheorem*{theoremA}{Theorem A}
\theoremstyle{definition}
\newtheorem{definition}[equation]{Definition}

\newtheorem{example}[equation]{Example}

\newcommand{\bZ}{\mathbb{Z}}

\newcommand{\nbd}{\nobreakdash}
\newcommand{\id}{\ensuremath{\mathrm{id}}}
\newcommand{\tensor}{\otimes}
\newcommand{\inv}{^{-1}}
\newcommand{\laurent}[1]{[#1,\, #1\inv]}
\newcommand{\nov}[1]{( \kern -1.7 pt ( {#1} ) \kern -1.7 pt )} 
\newcommand{\powers}[1]{[\kern -1.5pt [{#1}] \kern -1.5pt ]} 

\newcommand{\pp}{\mathbb{P}^1}
\newcommand{\qco}{\mathfrak{QCoh}}
\newcommand{\spec}{\mathrm{Spec}\,}
\newcommand{\invlim}{\mathop{\lim}\limits_\leftarrow{}}
\newcommand{\coker}{\mathrm{coker}}
\newcommand{\cone}{\mathrm{cone}}
\newcommand{\Hy}{\mathbf{H}}
\newcommand{\fpqc}{\text{fpqc}}
\newcommand{\Hfpqc}{\Hy_\fpqc}

\newcommand{\Rt}{R}

\usepackage[hypertexnames=false,colorlinks=false,pdfborderstyle={/S/U/W 1}]{hyperref}
\usepackage{bookmark}

\begin{document}

\title[Vector bundles on the projective line]{Vector bundles on the
  projective line\\and finite domination of chain complexes}

\date{\today}

\author{Thomas H\"uttemann}

\address{Thomas H\"uttemann\\ Queen's University Belfast\\ School of
  Mathematics and Physics\\ Pure Mathematics Research Centre\\ Belfast
  BT7~1NN\\ UK}

\email{t.huettemann@qub.ac.uk}

\urladdr{http://huettemann.zzl.org/}



\thanks{This work was supported by the Engineering and Physical
  Sciences Research Council [grant number EP/H018743/1]. The idea to
  write this paper was born during a research visit to Academia Sinica
  in July~2013; their support is gratefully acknowledged.}

\begin{abstract}
  Finitely dominated chain complexes over a \textsc{Laurent\/}
  polynomial ring in one indeterminate are characterised by vanishing
  of their \textsc{Novikov\/} homology. We present an
  algebro-geometric approach to this result, based on extension of
  chain complexes to sheaves on the projective line. We also discuss
  the $K$-theoretical obstruction to extension.
\end{abstract}

\maketitle

\thispagestyle{empty}

Let $R$ be a ring with unit, and let $\Rt \nov x = \Rt
\powers x [1/x]$ and $\Rt \nov{x\inv} = \Rt \powers {x\inv} [1/x\inv]$
be the rings of formal \textsc{Laurent\/} series (finite to the left
and right, respectively).

\begin{theoremA}[\textsc{Ranicki} {\cite[Theorem~2]{MR1341811}}]
  Let $C$ be a bounded complex of finitely generated free $\Rt
  \laurent x$\nbd-modules. The complex $C$ is $R$\nbd-finitely
  dominated (\ie, homotopy equivalent, as an $R$\nbd-module complex,
  to a bounded complex of finitely generated projective
  $R$\nbd-modules) if both
  \[C \tensor_{\Rt \laurent x} \Rt \nov{x} \quad \text{and} \quad C
  \tensor_{\Rt \laurent x} \Rt \nov{x\inv}\] are acyclic (and hence
  contractible) complexes.
\end{theoremA}

The homology of the complexes $C \tensor_{\Rt \laurent x} \Rt
\nov{x}$ and $C \tensor_{\Rt \laurent x} \Rt \nov{x\inv}$ is sometimes
referred to as the \textsc{Novikov\/} homology of~$C$; the theorem
states that if the \textsc{Novikov\/} homology of~$C$ is trivial, then
$C$~is $R$\nbd-finitely dominated. The converse holds as well as shown
in \cite{MR1341811}, see also \cite{MR2978500,MR3001337}. --- The
present paper focuses on two aspects of this result: An
algebro-geometric re-interpretation of the proof given by
\textsc{Ranicki}, and an analysis of the ``freeness'' hypothesis. The
latter discussion itself is motivated by the observation that {\it the
  theorem as given does not allow for iterative application to treat
  the case of \textsc{Laurent} polynomial rings $R [x, x\inv, y,
  y\inv] = R[y,y\inv][x,x\inv]$ in two indeterminates}.  Indeed, after
one application (using $\Rt \laurent y$ in place of~$R$) we are left
with a complex of {\it projective\/} modules rather than free
ones. This is relevant as there is a $K$\nbd-theoretical obstruction
to extending chain complexes to ``sheaves on the projective line'',
which prevents a direct application of the original proof to non-free
chain complexes. The obstruction can be expressed as an element of a
quotient of $K_{0} \big( \Rt \laurent x \big)$ isomorphic to $K_{-1}
(R) \oplus NK_{0}(R) \oplus NK_{0} (R)$, and is non-zero in general.

\section{Conventions, and rules of the game}

Throughout we let $R$ and~$S$ denote rings with unit. Modules are
right modules. Our complexes are homologically indexed: The
differential of a chain complex decreases the degree. If unspecified,
complexes are allowed to be unbounded both above and below.

In algebraic geometry, the category of quasi-coherent
$\mathcal{O}_{\spec (S)}$\nbd-modules on the affine scheme $\spec (S)$
is equivalent, {\it via\/} the global sections functor, to the
category of $S$\nbd-modules. We use this equivalence as a convenient
language in the case of non-commutative rings as well. For example,
irrespective of commutativity, we say that
\begin{itemize}
\item an $S$\nbd-module is a quasi-coherent sheaf on $\spec(S)$;
\item $T = \spec \Rt \laurent x$ is the algebraic torus of
  dimension~$1$ over~$R$;
\item $P = \spec R$ is the point over~$R$;
\item the canonical map $p \colon T \rTo P$ induces a push-forward
  functor $p_{*}$ which assigns to each quasi-coherent sheaf~$M$ on~$T$ a
  quasi-coherent sheaf $p_{*} M$ on~$P$, and a pull-back
  functor~$p^{*}$.
\end{itemize}
The last point says, in module theoretic terms, that an $\Rt \laurent
x$\nbd-module~$M$ can be considered as an $R$\nbd-module $p_{*}M$ by
restriction of scalars, and that an $R$\nbd-module $N$ gives rise to
an induced $\Rt \laurent x$\nbd-module given by $p^{*}N = N
\tensor_{R} \Rt \laurent x$.

\section{Finite domination}

\begin{definition}
  A chain complex of $S$\nbd-modules, or quasi-coherent sheaves
  on~$\spec(S)$, is called
  \begin{enumerate}
  \item a {\it strict perfect complex\/} if it is bounded and consists
    of finitely generated projective $S$\nbd-modules;
  \item an {\it $S$\nbd-perfect complex\/} if it is quasi-isomorphic
    to a strict perfect complex;
  \item a {\it complex of vector bundles on $\spec S$} if it consists
    of finitely generated projective $S$\nbd-modules;
  \item a {\it complex of trivial vector bundles on $\spec S$\/} if it
    consists of finitely generated free $S$\nbd-modules;
  \item an {\it $S$\nbd-finitely dominated complex\/} if it is chain
    homotopy equivalent to a strict perfect complex.
  \end{enumerate}
\end{definition}

We will need the following trivial observation: {\it Every strict
  perfect complex~$C$ over $\spec(S)$ is a direct summand of a bounded
  complex of trivial vector bundles over $\spec(S)$; the
  complement~$C'$ is a strict perfect complex and can be chosen to
  have trivial differential.} Indeed, choose $C'_{n}$ to be a finitely
generated projective module so that $C_{n} \oplus C'_{n}$ is finitely
generated free (with $C'_{n}=0$ whenever $C_{n} = 0$) and equip $C'$
with zero-differentials.

\begin{proposition}[Characterisations of finite domination]
  \label{prop:fin_dom}
  The following five statements are equivalent for a bounded below
  chain complex~$C$ of projective $S$\nbd-modules:
  \begin{enumerate}[{\rm (1)}]
  \item The complex $C$ is $S$\nbd-perfect.
  \item The complex $C$ is $S$\nbd-finitely dominated.
  \item There exist a bounded complex~$D$ of finitely generated free
    $S$\nbd-modules and chain maps $r \colon D \rTo C$ and $s \colon C
    \rTo D$ together with a chain homotopy $r \circ s \simeq \id_{C}$.
  \item There exist a strict perfect complex~$D$ of $S$\nbd-modules
    and chain maps $r \colon D \rTo C$ and $s \colon C \rTo D$
    together with a chain homotopy $r \circ s \simeq \id_{C}$.
  \item The complex $C$ is a direct summand of an $S$\nbd-finitely
    dominated bounded below complex~$E$.
  \end{enumerate}
\end{proposition}

\begin{proof}
  The equivalence of~(1) and~(2) follows from the fact that any
  quasi-isomorphism between bounded below complexes of projective
  modules is a homotopy equivalence. The equivalence of~(2) and~(3)
  has been shown by \textsc{Ranicki}
  \cite[Proposition~3.2~(ii)]{MR815431}. The equivalence of (3)
  and~(4) follows from the observation above that every strict perfect
  complex is a direct summand of a bounded complex of trivial vector
  bundles.  If (4) holds then $C$~is a direct summand of the mapping
  cylinder~$Z$ of~$s$ which is homotopy equivalent to~$D$ and hence
  finitely dominated (the requisite projection map $Z \rTo C$ is
  determined by the homotopy $r \circ s \simeq \id_{C}$), so that~(5)
  holds. Conversely, given~(5) there exists a strict perfect
  complex~$D$ homotopy equivalent to~$E$. Then the composition of $C
  \rTo E \simeq D$ and $D \simeq E \rTo C$ is homotopic to~$\id_{C}$
  so that (4)~holds.
\end{proof}

\section{The basic homological setup}
\label{sec:basic_hom}

In the sequel we will repeatedly need the following homological
constructions related to diagrams (in some category of modules) of the
shape
\begin{equation}
  \label{eq:shape}
  \dd M = \, \big( M^{-} \rTo^{\mu^{-}} M \lTo^{\mu^{+}} M^{+}  \big) \ .
\end{equation}
(A map of such diagrams is a triple of maps $\phi = (f^{-}, f, f^{+})$
compatible with the structure maps.) --- For a diagram of
$S$\nbd-modules of shape~\eqref{eq:shape}, we let $H(\dd M)$ denote
the chain complex $M \lTo^{-\mu^{-}+\mu^{+}} M^{-} \oplus M^{+}$
concentrated in chain degrees $-1$ and~$0$; we write $H^{q} (\dd M)$
for the $(-q)$th homology module of~$H(\dd M)$ call this the {\it
  $q${\rm th} cohomology module of~$\dd M$}. So $H^{0} (\dd M) = \ker
(-\mu^{-}+ \mu^{+})$ and $H^{1} (\dd M) = \coker (-\mu^{-}+ \mu^{+})$,
while $H^{q} (\dd M) = 0$ for $q \neq 0,1$. It can be shown that
$H^{q} (\dd M) \iso \invlim^{q} (\dd M)$ for all~$q$.

More generally, for a diagram of $S$\nbd-module chain
complexes~\eqref{eq:shape} we denote by $\Hy(\dd M)$ the totalisation
of the double complex $H(\dd M)$. More explicitly, we have $\Hy (\dd
M)_{n} = M^{-}_{n} \oplus M^{+}_{n} \oplus M_{n+1}$ with differential
given by
\[(a^{-},\, a,\, a^{+}) \mapsto \big( d^{-} (a^{-}),\, d^{+}
(a^{+}),\, -\mu^{-}(a^{-}) + \mu^{+}(a^{+}) - d(a) \big) \ ,\] where
$d$, $d^{+}$ and~$d^{-}$ are the differentials of the chain complexes
$M$, $M^{+}$ and~$M^{-}$, respectively. The complex $\Hy(\dd M)$ is
called the {\it hypercohomology chain complex of~$\dd M$}, and the
$(-q)$th homology module of~$\Hy(\dd M)$ is the {\it hypercohomology
  module $\Hy^{q}(\dd M)$}.

A map $\phi = (f^{-}, f, f^{+})\colon \dd M \rTo \dd N$ of diagrams of
chain complexes induces a map $\phi_{*} \colon \Hy (\dd M) \rTo \Hy
(\dd N)$. We also have a chain complex inclusion $\iota \colon
H^{0}(\dd M) \rTo \Hy(\dd M)$ where the source is the chain complex
obtained by applying the functor $H^{0}$ to~$\dd M$ in each chain
level.

\begin{lemma}
\label{lem:H1-principle}
If the components of~$\phi$ are all quasi-isomorphisms, then
$\phi_{*}$ is a quasi-isomorphism. If $H^{1}(\dd M) = 0$ (levelwise
application of~$H^{1}$) then the map $\iota$ is a quasi-isomorphism.
\end{lemma}

\begin{proof}
  The first claim follows easily from the five lemma and the existence
  of a short exact sequence
  \[0 \rTo M[1] \rTo \Hy (\dd M) \rTo M^{-} \oplus M^{+} \rTo 0 \]
  natural with respect to maps of diagrams~$\phi$. --- For the second
  claim, note that the hypothesis $H^{1}(\dd M) = 0$ translates into
  the sequence
  \[0 \lTo M \lTo^{-\mu^{-}+\mu^{+}} M^{-} \oplus M^{+} \lTo^{\iota}
  H^{0} (\dd M) \lTo 0\] being exact (in each chain level) so that
  $H^{0}(\dd M)$ is quasi-isomorphic to the ``homotopy fibre'' of the
  map $-\mu^{-} + \mu^{+}$ (just as $M$ is then quasi-isomorphic to
  the mapping cone of~$\iota$). The complex $\Hy (\dd M)$ is a model
  for the homotopy fibre.
\end{proof}

\section{The projective line}
\label{sec:p1}

We will now define categories of modules on the projective line $\pp$
over~$R$, in the spirit of \textsc{Bass} \cite[\S{}XII.9]{MR0249491}.

\begin{definition}[Quasi-coherent sheaves and vector bundles]\
  \label{def:sheaf}
  \begin{enumerate}
  \item A {\it quasi-coherent sheaf on~$\pp$}, or just {\it sheaf\/}
    for short, is a diagram~$\dd M$ of the form~\eqref{eq:shape}
    such that
    \begin{itemize}
    \item the entries are modules over the rings $\Rt[x\inv]$, $\Rt
      \laurent x$ and $\Rt[x]$, respectively;
    \item the maps $\mu^{-}$ and~$\mu^{+}$ are $\Rt[x\inv]$\nbd-linear
      and $\Rt[x]$\nbd-linear, respectively;
    \item both adjoint maps
      \[M^{-} \tensor_{\Rt[x\inv]} \Rt \laurent x \rTo M \lTo M^{+}
      \tensor_{\Rt[x]} \Rt \laurent x\] are isomorphisms of $\Rt
      \laurent x$\nbd-modules.
    \end{itemize}
  \item The category of quasi-coherent sheaves is denoted
    $\qco(\pp)$. Morphisms are triples $(f^{-},\, f,\, f^{+})$ of
    linear maps compatible with the structure maps.
  \item The sheaf \eqref{eq:shape} is a {\it vector bundle\/}
    if all its constituent modules are finitely generated projective
    over their respective ground rings.
  \end{enumerate}
\end{definition}

It might be useful to recall here that $\pp$ has a
\textsc{Zariski}-open cover consisting of the affine lines $U^{+} =
\spec R[x]$ and $U^{-} = \spec R[x\inv]$, with intersection $U^{-} \cap
U^{+} = T = \spec R \laurent x \subset \pp$.

\begin{definition}
  Given a sheaf~$\dd M$ as in~\ref{def:sheaf} and an integer~$n$, we
  denote by $\dd M(n)$ any sheaf of the form $M^{-} \rTo^{x^{k}
    \mu^{-}} M \lTo^{x^{-\ell} \mu^{+}} M^{+}$ with $k,\ell \in \bZ$
  such that $k + \ell = n$ (all these sheaves are isomorphic in
  $\qco(\pp)$). We call $\dd M(n)$ the {\it $n${\rm th} twist of~$\dd
    M$}.
\end{definition}

In the context of sheaves, we will write $H(\pp; \dd M)$ for $H(\dd
M)$ to emphasise the similarity to algebraic geometry (in fact, the
modules $H^{q} (\pp; \dd M)$ are sheaf cohomology modules, computed
from a \textsc{\v Cech} complex, in case $R$ is a commutative
ring). We similarly use the notation $\Hy(\pp; \dd M)$ to denote
$\Hy(\dd M)$ for a chain complex~$\dd M$ of sheaves.

\begin{example}
  \label{example:H_of_On}
  The $n$th {\it twisting sheaf}, denoted $\mathcal{O}(n)$, is a
  vector bundle of the form
  \[\Rt[x \inv] \rTo[l>=3em]^{x^{k}} \Rt \laurent x
  \lTo[l>=3em]^{x^{-\ell}} \Rt[x] \] with $k, \ell \in \bZ$ such that
  $k + \ell = n$; the structure maps are inclusions followed by
  multiplication.  For $q,n \in \bZ$ we have
  \[H^{q} \big (\pp; \mathcal{O}(n) \big) \iso
  \begin{cases}
    R^{n+1} & \text{for \(q=0\) and \(n \geq 0\);} \\
    R^{-n-1} & \text{for \(q=1\) and \(n \leq -2\);} \\
    0 & \text{otherwise.}
  \end{cases}\] In fact, the $H^{q} \big( \pp; \mathcal{O}(n) \big)$
  can be identified with free $R$\nbd-submodules of the $R$\nbd-module
  $\Rt \laurent x$; bases are $\{x^{-\ell}, x^{-\ell+1}, \cdots,
  x^{k}\}$ in the first case, and $\{x^{k+1}, x^{k+2}, \cdots,
  x^{-\ell-1}\}$ in the second.
\end{example}

\begin{lemma}[Extending morphisms from~$T$ to~$\pp$]
  \label{lem:extend_map}
  Suppose we have two sheaves $\dd Z = (Z^{-} \rTo^{\zeta^{-}} Z
  \lTo^{\zeta^{+}} Z^{+})$ and $\dd Y = (Y^{-} \rTo^{\upsilon^{-}} Y
  \lTo^{\upsilon^{+}} Y^{+})$ on~$\pp$, and a homomorphism $f \colon Z
  = \dd Z|_{T}\rTo \dd Y|_{T} = Y$. Suppose that $Z^{-}$ and~$Z^{+}$
  are finitely generated, and that $\upsilon^{-}$ and~$\upsilon^{+}$
  are injective (equivalently, $Y^{\pm}$ have no $x^{\pm
    1}$\nbd-torsion). Then there exist integers $k, \ell \geq 0$ and
  homomorphisms $f^{\pm} \colon Z^{\pm} \rTo Y^{\pm}$ fitting into a
  commutative diagram
  \begin{diagram}[s=1.1cm]
    Z^- & \rTo^{\zeta^{-}} & Z & \lTo^{\zeta^{+}} & Z^+\\
    \dTo>{f^-} && \dTo>f && \dTo>{f^+} \\
    Y^- & \rTo^{x^k \upsilon^{-}} & Y & \lTo^{x^{-\ell} \upsilon^{+}}
    & Y^+
  \end{diagram}
  so that $f$ extends to a map of sheaves $(f^{-}, f, f^{+}) \colon
  \dd Z \rTo \dd Y(k+\ell)$.
\end{lemma}

\begin{proof}
  Choose an epimorphism $\epsilon^{+} \colon F^{+} \rTo Z^{+}$, with
  $F^{+}$ finitely generated free. The composition $f \circ \zeta^{+}
  \circ \epsilon^{+}$ factors as $F^{+} \rTo^{\phi^{+}} Y^{+}
  \rTo^{x^{-\ell} \upsilon^{+}} Y$, for some~$\ell \geq 0$, as $Y \iso
  Y^{+} \tensor_{\Rt [x]} \Rt \laurent x$ is the sequential colimit of
  \begin{displaymath}
    Y^{+} \rTo^{x} Y^{+} \rTo^{x} \cdots \ .
  \end{displaymath}
  Now the image of $\ker (\epsilon^{+})$ under~$\phi^{+}$ in~$Y^{+}$
  is trivial; indeed, the map $x^{-\ell} \upsilon^{+}$ is injective,
  so we can check triviality by post-composing with this map. But
  then, by construction, we are reduced to showing that the image of
  $\ker(\epsilon^{+})$ in~$Y$ under the map $x^{-\ell} \upsilon^{+}
  \circ \phi^{+} = f \circ \zeta^{+} \circ \epsilon^{+}$ is trivial,
  which is clear. It follows that $\phi^{+}$ induces a map $f^{+}
  \colon Z^{+} = F^{+}/\ker(\epsilon^{+}) \rTo Y^{+}$ which, when
  post-composed with~$x^{-\ell}\upsilon^{+}$, agrees with $f \circ
  \zeta^{+}$. --- The construction of $f^{-}$ follows the symmetric
  procedure.
\end{proof}

\section{An algebro-geometric approach to Theorem~A}
\label{sec:criterion}

We write $\Rt \powers x$ for the ring of formal power series in~$x$;
its localisation $\Rt \nov x = \Rt \powers{x}_{x}$ by~$x$ is the ring
of formal \textsc{Laurent\/} series in~$x$. We have the variants $\Rt
\powers {x\inv}$ and $\Rt \nov{x \inv}$ of formal power series and
formal \textsc{Laurent} series in~$x\inv$. Clearly $\Rt \nov x \cap
\Rt \nov{x \inv} = \Rt \laurent x$.

The relevant observation for us is that $N^{+} = \spec \Rt \powers x$
is an infinitesimal (formal) neighbourhood of $0 \in U^{+} \subset
\pp$, and that $\spec \Rt \nov x$ corresponds to $N^{+} \setminus
\{0\}$. Replacing $x$ by $x\inv$ gives infinitesimal neighbourhoods of
$\infty \in U^{-} \subset \pp$, with and without the point $\infty$
included.

\begin{definition}[Strict perfect and perfect complexes]
  A chain complex $\dd C = (C^{-} \rTo C \lTo C^{+})$ of sheaves is
  called {\it strict perfect\/} if it is a bounded complex of vector
  bundles. We call $\dd C$ {\it perfect\/} if it is connected by a
  chain of quasi-isomorphisms in~$\qco(\pp)$ to a strict perfect
  complex. Here a {\it quasi-isomorphism\/} is a map of complexes of
  sheaves such that all three of its constituent maps of chain
  complexes of modules are quasi-isomorphisms.
\end{definition}

\begin{theorem}[Algebro-geometric reformulation of Theorem~A]
  \label{thm:reformulation}
  Let $C$ be a bounded complex of trivial vector bundles on the
  algebraic torus~$T$. The push-forward $p_{*} C$ is an
  $R$\nbd-perfect complex if $C$ is homologically trivial
  infinitesimally near~$0$ and~$\infty$ (that is, on~$N^{+} \setminus
  \{0\}$ and~$N^{-} \setminus \{\infty\}$).
\end{theorem}

We begin the proof by observing that bounded complexes of trivial
vector bundles on~$T$ extend to strict perfect complexes on~$\pp$:

\begin{proposition}[Extending complexes of trivial vector bundles]
  \label{prop:extend_trivial}
  Given a bounded chain complex $C$ of trivial vector bundles on~$T$,
  there exists a strict perfect complex $\dd V =
  (V^{-} \rTo V \lTo V^{+})$ of sheaves such that
  \begin{enumerate}[{\rm (1)}]
  \item the complexes $V^{-}$ and~$V^{+}$ consist of finitely
    generated free modules;
  \item the restriction $V = \dd V|_{T}$ is isomorphic to the complex~$C$;
  \item the complex of global sections $H^{0}(\pp; \dd V) = \invlim
    \dd V$ is a bounded complex of finitely generated free
    $R$\nbd-modules;
  \item the higher cohomology modules $H^{j}(\pp; \dd V) = \invlim^{j}
    \dd V$ are trivial for all $j \geq 1$ in each chain degree.
  \end{enumerate}
\end{proposition}

In fact, given a finitely generated free $R \laurent x$\nbd-module $M$
of rank\footnote{By rank we mean the cardinality of some basis of~$M$;
  this number might not be uniquely determined by~$M$ unless $R$ has
  the invariant basis property}~$r$, the sheaf $\bigoplus_{r}
\mathcal{O}(n)$, for any $n \in \bZ$, has the property that its
restriction to~$T$ is isomorphic to~$M$. The main point is that the
differentials can be lifted to~$\pp$ as well, by downward induction on
the chain degree: Assuming $\dd V_{m+1}$ has been constructed, set
$\dd Y = \bigoplus_{r} \mathcal{O}(0)$ for $r$ a rank of~$C_{m}$,
and apply Lemma~\ref{lem:extend_map} to $\dd Z = \dd V_{m+1}$ and $f
= \partial \colon V_{m+1} \rTo V_{m} \iso \big(\Rt \laurent
x\big)^{r}$. Finally set $\dd V_{m} = \dd Y(k+\ell)$.  Since the
structure maps of~$\mathcal{O}(n)$ are injective, it is easy to check
that the extension of~$f$ to~$\pp$ is a differential. The Proposition
now follows from the calculations in Example~\ref{example:H_of_On}.

\medbreak

Back to Theorem~\ref{thm:reformulation}, let us thus extend our
complex~$C$ to a strict perfect complex $\dd C = (C^{-} \rTo^{\mu^{-}}
C \lTo^{\mu^{+}} C^{+})$ as described in
Proposition~\ref{prop:extend_trivial}. From
Lemma~\ref{lem:H1-principle} we know that the map $\iota \colon H^{0}
(\pp; \dd C) \rTo \Hy (\pp; \dd C)$ is a quasi-isomorphism; note that
the source of this map is a bounded complex of finitely generated free
$R$\nbd-modules, by construction.

Observe that for $R$ commutative and \textsc{noether}ian, the rings $R
\powers x$ and $R \laurent x$ are flat over~$R[x]$ so that $\spec R
\powers x \amalg \spec R \laurent x \rTo \spec R[x]$ is a covering of
the affine line in the fpqc topology; furthermore, the intersection of
the two covering sets is $\spec R \nov x$. So for general~$R$, we
represent the restriction~$C^{+} = \dd C|_{U^{+}}$ of~$\dd C$ to
$U^{+} = \spec \Rt[x]$ in the fpqc topology by the diagram
\begin{equation}
  \label{eq:c+fpqc}
  \dd C^{+}_{\fpqc} = \, \Big( C^{+} \mathop{\tensor}_{\Rt[x]} \Rt \powers
  x \rTo C^{+} \mathop{\tensor}_{\Rt[x]} \Rt \nov x \lTo C^{+}
  \tensor_{\Rt [x]} \Rt \laurent x \Big) \ .
\end{equation}
Since $C^{+}$ consists of free $\Rt[x]$\nbd-modules, and since the
$\Rt[x]$\nbd-module diagram $\Rt \powers x \rTo \Rt \nov x \lTo R
\laurent x$ has $H^{0} = R[x]$ and $H^{1} = 0$, the first map in the
composition $\kappa^{+}$ of canonical maps
\[C^{+} \rTo^{\iso} H^{0}_{\fpqc} (U^{+}; \dd C^{+}_{\fpqc})
\rTo^{\simeq}_{\iota} \Hfpqc (U^{+}; \dd C)\] is an isomorphism, where
the source of~$\iota$ stands for the chain complex $H^{0} (\dd
C^{+}_{\fpqc})$ and the target for $\Hy (\dd
C^{+}_{\fpqc})$. Moreover, the structure isomorphism $C^{+}
\tensor_{\Rt [x]} \Rt \laurent x \iso C$ of~$\dd C$ results in a chain
map $\lambda^{+} \colon \Hfpqc (U^{+}; \dd C) \rTo C$ such that
$\mu^{+} = \lambda^{+} \circ \kappa^{+}$.

We have an analogous quasi-isomorphism $\kappa^{-} \colon C^{-} \rTo
\Hfpqc (U^{-}; \dd C) = \Hy (\dd C^{-}_{\fpqc})$, and a factorisation
$\mu^{-} = \lambda^{-} \circ \kappa^{-} \colon C^{-} \rTo C$.

Now note that by hypothesis the chain complex
\[C^{+} \mathop{\tensor}_{\Rt[x]} \Rt \nov x \iso C^{+}
\mathop{\tensor}_{\Rt[x]} \Rt \laurent x \mathop{\tensor}_{\Rt
  \laurent x} \Rt \nov x \iso C \mathop{\tensor}_{\Rt \laurent x} \Rt
\nov x \] is acyclic. This means we can replace the middle entry
in~\eqref{eq:c+fpqc} by the trivial chain complex; by
Lemma~\ref{lem:H1-principle}, the resulting map
\begin{multline*}
  \qquad \Hfpqc (U^{+}; \dd C) \rTo \Hy \big( C^{+} \tensor_{\Rt[x]} \Rt
  \powers x \rTo 0 \lTo C \big ) \\ = \big( C^{+} \tensor_{\Rt[x]} \Rt
  \powers x \big) \oplus C \qquad
\end{multline*}
is a quasi-isomorphism. We have a similar quasi-isomorphism,
constructed in a similar manner, $\Hfpqc (U^{-}; \dd C) \rTo \big(
C^{-} \tensor_{\Rt[x\inv]} \Rt \powers {x\inv} \big) \oplus C$.

All these maps fit into the following commutative diagram:
\begin{diagram}[s=1cm]
  C^{-} & \rTo^{\mu^{-}} & C & \lTo^{\mu^{+}} & C^{+} \\
  \dTo>{\kappa^{-}}<{\simeq} && \dTo<{=} && \dTo>{\kappa^{+}}<{\simeq}
  \\
  \Hfpqc (U^{-}; \dd C) & \rTo^{\lambda^{-}} & C & \lTo^{\lambda^{+}}
  & \Hfpqc (U^{+}; \dd C) \\
  \dTo<{\simeq} && \dTo<{=} && \dTo<{\simeq} \\
  \big( C^{-} \tensor_{R[x \inv]} R \powers x\inv \big) \oplus C &
  \rTo^{\text{pr}_{C}} & C & \lTo^{\text{pr}_{C}} & \big( C^{+}
  \tensor_{R[x]} R \powers x \big) \oplus C
\end{diagram}
Application of~$\Hy$ to each row then results in a quasi-isomorphism
from $\Hy (\pp; \dd C)$ (the first row) with a chain complex that
contains, by construction, the complex $\Hy (C \rTo^{=} C \lTo^{=} C)$
as a direct summand. Now the latter is quasi-isomorphic to~$C$, by
Lemma~\ref{lem:H1-principle}, and the former is quasi-isomorphic to a
bounded complex of finitely generated free $R$\nbd-modules, as noted
before. It follows from Proposition~\ref{prop:fin_dom} that $C$~is
perfect as claimed.

\section{Extending sheaves from the torus to the projective line}
\label{sec:extending}

The proof given in the previous section relies on our ability to
extend the given chain complex~$C$ to a complex of vector bundles
on~$\pp$. There is a $K$\nbd-theoretical obstruction to extension, as
described in this section (which is motivated by the discussion of
extension problems by \textsc{Thomason} and \textsc{Trobaugh}
\cite[\S5.5]{MR1106918}). As a matter of terminology, we say that a
perfect complex~$C$ of $\Rt \laurent x$\nbd-modules {\it extends
  to~$\pp$ up to quasi-isomorphism\/} if there exists a strict perfect
complex $\dd V$ of sheaves on~$\pp$ such that $\dd V|_{T}$~is
quasi-isomorphic to~$C$.

\begin{lemma}
  \label{lem:extend_2_of_3}
  Let $0 \rTo C_{1} \rTo^{\alpha} C_{2} \rTo^{\beta} C_{3} \rTo 0$ be
  a short exact sequence of perfect complexes of $\Rt \laurent
  x$\nbd-modules. Suppose that two of the three complexes extend
  to~$\pp$ up to quasi-isomorphism. Then so does the third.
\end{lemma}

\begin{proof}
  Suppose first that $C_{1}$ and~$C_{2}$ extend up to
  quasi-isomorphism. Let $\nu \colon C_{1} \rTo C_{2}$ be any chain
  map; we will show that then the mapping cone of~$\nu$ extends
  to~$\pp$ up to quasi-isomorphism. --- By hypothesis we find strict
  perfect complexes $\dd V_{1}$ and $\dd V_{2}$ on~$\pp$ with $\dd
  V_{1}|_{T} \simeq C_{1}$ and $\dd V_{2}|_{T} \simeq C_{2}$. Together
  with $\nu$ these quasi-isomorphisms determine a morphism $\dd
  V_{1}|_{T} \rTo \dd V_{2}|_{T}$ in the derived category of~$\Rt
  \laurent x$. Since $\dd V_{1}|_{T}$ is strict perfect, we can
  represent this morphism by an actual chain map~$\omega$. Note that
  the mapping cone of~$\omega$ is quasi-isomorphic to the mapping cone
  of~$\nu$.

  After replacing $\dd V_{2}$ by a twist $\dd V_{2}(k)$ for
  sufficiently large $k > 0$, we can extend~$\omega$ to a map~$\Omega$
  of complexes of sheaves on~$\pp$, which satisfies $\Omega|_{T} =
  \omega$; this follows from Lemma~\ref{lem:extend_map}, applied to
  $\dd Z=\dd V_{1}$, $\dd Y=\dd V_{2}$ and $f = \omega$ in each chain
  level (the constructed extensions are automatically compatible with
  the differentials). The mapping cone of~$\Omega$ is then a strict
  perfect complex of sheaves on~$\pp$ which restricts to the mapping
  cone of~$\omega$ on~$T$, thus extending the mapping cone of~$\nu$ up
  to quasi-isomorphism.

  \smallbreak

  The theorem now follows easily in view of the quasi-iso\-mor\-phisms
  $C_{3} \simeq \mathrm{Cone}(\alpha)$, $C_{1} \simeq \cone(\beta)[-1]$
  and $C_{2} \simeq \cone \big( \cone(\alpha) \rTo C_{1}[1] \big)
  [-1]$.
\end{proof}

Let $K_{0} (S)$ denote the \textsc{Grothendieck} group of isomorphism
classes of finitely generated projective $S$\nbd-modules. That is,
$K_{0}(S)$ is the quotient of the free \textsc{abel}ian group
generated by the isomorphism classes $[P]$, for $P$ a finitely
generated projective $S$\nbd-module, by the subgroup with one
generator $[P]+[R] - [Q]$ for every short exact sequence
\begin{equation}
  \label{eq:ses}
  0 \rTo P \rTo Q \rTo R \rTo 0
\end{equation}
of finitely generated projective $S$\nbd-modules. It follows from
general $K$\nbd-theory, specifically from 1.5.6, 1.11.2 and 1.11.7
of~\cite{MR1106918}, that we could have defined $K_{0}(A)$ as the
quotient of the free \textsc{abel}ian group generated by
quasi-isomorphism classes of strict perfect $S$\nbd-module complexes
by the subgroup with one generator $[P] + [R] - [Q]$ for each short
exact sequence~\eqref{eq:ses} of strict perfect complexes which splits
in each chain level.

By replacing ``$S$\nbd-module'' by ``sheaf on~$\pp$'', and ``finitely
generated projective $S$\nbd-module'' by ``vector bundle on~$\pp$'' we
obtain two equivalent definitions of $K_{0}(\pp)$.

It might be worthwhile to point out that in the case of
$S$\nbd-modules, it is automatic that the sequence~\eqref{eq:ses}
splits, and that in the case of strict perfect complexes
of $S$\nbd-modules the sequence~\eqref{eq:ses} splits in each chain
level. Also, two strict perfect complexes of $S$\nbd-modules are
quasi-isomorphic if and only if they are chain homotopy equivalent, so
we could have replaced ``quasi-isomorphism class'' by ``homotopy
class''. The situation is entirely different in the case of sheaves
on~$\pp$ where, in general, a sequence~\eqref{eq:ses} will be
non-split, and where quasi-isomorphism does not imply homotopy
equivalence.

\begin{theorem}
  \label{thm:vb_extends}
  Let $C$ be a strict perfect complex of $R \laurent
  x$\nbd-modules. The following three statements are equivalent:
  \begin{enumerate}[{\rm (1)}]
  \item The complex~$C$ extends to~$\pp$ up to quasi-isomorphism.
  \item The class $[C] \in K_{0} \big( \Rt \laurent x \big)$ lies in
    the image of the pullback map
    \[p^{*} \colon K_{0} (R) \rTo K_{0} \big( \Rt \laurent x \big) \ ,
    \quad [M] \mapsto \big[ M \tensor_{R} \Rt \laurent x \big]\] which
    is induced by the map $p \colon T \rTo P$.\goodbreak
  \item The class $[C] \in K_{0} \big( \Rt \laurent x \big)$ lies in
    the image of the restriction map
    \[i^{*} \colon K_{0} (\pp) \rTo K_{0} \big( \Rt \laurent x \big) \
    , \quad [C^{-} \rTo C \lTo C^{+}] \mapsto [C]\] which is induced
    by the inclusion map $i \colon T \subset \pp$.
  \end{enumerate}
\end{theorem}

\begin{proof}
  We show first that statements~(1) and~(3) are equivalent. The
  implication (1)~$\Rightarrow$~(3) is trivial.

  A bounded complex of finitely generated free $\Rt\laurent x$-modules
  extends to~$\pp$, as observed in
  Proposition~\ref{prop:extend_trivial}.  Next, for $C$ an arbitrary
  strict perfect complex we can find a bounded complex~$C'$ of
  finitely generated projective $R \laurent x$\nbd-modules such that
  $C \oplus C'$ consists of finitely generated free modules, and thus
  extends to~$\pp$.

  Now let $\Pi$ denote free \textsc{abel}ian group generated by the
  quasi-isomorphism classes $\langle A \rangle$ of strict perfect
  complexes~$A$ of $R \laurent x$\nbd-modules, and let $\Lambda$ be the
  subgroup generated by the relations
  \begin{subequations}
    \begin{gather}
      \langle A\rangle + \langle B\rangle - \langle A \oplus B \rangle
      \ , \label{rel:sum}\\
      \langle K \rangle \text{ for \(K\) extending to~\(\pp\) up to
        quasi-isomorphism.}
    \end{gather}
  \end{subequations}
  Define $\pi = \Pi/\Lambda$; we denote the image of~$\langle C
  \rangle$ in~$\pi$ by $[C]$. For $C$ and~$C'$ as in the previous
  paragraph the relations clearly imply $[C] + [C'] = 0$ so that
  $C'$~represents in inverse of~$[C]$.

  Now suppose that $C$ is such that $[C] = 0$ in $\pi = \Pi /
  \Lambda$. Then we must have $\langle C \rangle \in \Lambda$, that
  is, we find finitely many $A_{i}$, $B_{i}$, $\bar A_{j}$ and~$\bar B_{j}$,
  and finitely many $K_{k}$ and~$L_{\ell}$ that extend to~$\pp$ up to
  quasi-isomorphism, such that
  \begin{multline*}
    \langle C \rangle + \sum_{j} \big( \langle \bar A_{j} \rangle +
    \langle \bar B_{j} \rangle - \langle \bar A_{j} \oplus \bar B_{j} \rangle
    \big) + \sum_{\ell} \langle L_{\ell} \rangle \\
    = \sum_{i} \big( \langle A_{i} \rangle + \langle B_{i} \rangle -
    \langle A_{i} \oplus B_{i} \rangle \big) + \sum_{k} \langle K_{k}
    \rangle \ .
  \end{multline*}
  That is, we have the equality
  \begin{multline*}
    \langle C \rangle + \sum_{j} \big( \langle \bar A_{j} \rangle +
    \langle \bar B_{j} \rangle \big) + \sum_{i} \langle A_{i} \oplus B_{i}
    \rangle + \sum_{\ell} \langle L_{\ell} \rangle \\
    = \sum_{i} \big( \langle A_{i} \rangle + \langle B_{i} \rangle
    \big) + \sum_{j} \langle \bar A_{j} \oplus \bar B_{j} \rangle
    + \sum_{j} \langle K_{j} \rangle
  \end{multline*}
  in the free \textsc{abel}ian group~$\Pi$; this implies that the
  quasi-isomorphism classes occurring on both sides must agree,
  counted with multiplicities. That is, we have a quasi-isomorphism $C
  \oplus A \oplus L \simeq A \oplus K$ where we write $A =
  \bigoplus_{i} \big( A_{i} \oplus B_{i} \big) \oplus \bigoplus_{j}
  (\bar A_{j} \oplus \bar B_{j})$, $L = \bigoplus_{\ell} L_{\ell}$ and
  $K = \bigoplus_{k} K_{k}$. Let $A'$ represent an inverse for~$A$
  in~$\pi$ as in the first paragraph of this proof; then $A' \oplus A$
  extends to~$\pp$ up to quasi-isomorphism so that $A' \oplus A \oplus
  L$ and $C \oplus A' \oplus A \oplus L \simeq A' \oplus A \oplus K$
  do too. It follows from Lemma~\ref{lem:extend_2_of_3} that $C$
  extends to~$\pp$ up to quasi-isomorphism.

  \medskip

  We now know that $[C]=0 \in \pi$ if and only if $C$ extends to~$\pp$
  up to quasi-isomorphism. It remains to observe that $\pi =
  \coker(i^{*})$. This is almost immediate from the presentation of
  $K$-groups given at the beginning of this section; the one
  discrepancy we have is that $\pi$ incorporates the
  relation~\eqref{rel:sum} for split short exact sequences only,
  whereas in~$K_{0}(\Rt \laurent x)$ we have a relation for all
  levelwise split sequences. It is enough to verify the following
  assertion: {\it Given a short exact sequence $0 \rTo C \rTo^{f} D
    \rTo^{g} E \rTo 0$ of strict perfect complexes of $\Rt \laurent
    x$\nbd-modules, we have $[C] + [E] = [D]$ in the group~$\pi$.} For
  this, let $C'$ and~$E'$ represent inverses of~$[C]$ and~$[E]$
  in~$\pi$, as in the beginning of this proof. Then we have a short
  exact sequence
  \[0 \rTo C \oplus C' \rTo^{(f, \id_{C'}, 0)} D \oplus C' \oplus E'
  \rTo^{(g, \id_{E'})} E \oplus E' \rTo 0\] of strict perfect
  complexes, with first and third term being free; it follows that $D
  \oplus C' \oplus E'$ extends to~$\pp$ up to quasi-isomorphism
  (Lemma~\ref{lem:extend_2_of_3}) and thus represents $0 \in
  \pi$. Consequently, $[D] = -[C'] - [E'] = [C] + [E] \in \pi$.

  Finally, we need to prove that statements~(2) and~(3) are
  equivalent. As shown by \textsc{Bass} \cite[\S{}XII.9]{MR0249491}
  and \textsc{Quillen} \cite[\S8.3]{MR0338129} there is an isomorphism
  $K_{0}(R) \times K_{0} (R) \rTo^{\iso} K_{0} (\pp)$ which sends the
  element $\big([M], [N] \big)$ to~$[M \tensor_{R} \mathcal{O}(0)] +
  [M \tensor_{R} \mathcal{O} (-1)]$, which is mapped by~$i^{*}$ to
  $p^{*} \big([M] \big) + p^{*} \big([N] \big)$. It follows that the
  image of~$i^{*}$ coincides with the image of the map~$p^{*}$.
\end{proof}

The map~$p^{*}$ exhibits $K_{0}(R)$ as a direct summand of $K_{0}
\big( \Rt \laurent x \big)$ with complement $\coker(p^{*}) \iso K_{-1}
(R) \oplus NK_{0}(R) \oplus NK_{0} (R)$, according to the
\textsc{Bass}-\textsc{Heller}-\textsc{Swan} formula
\cite[Corollary~XII.7.6]{MR0249491}. Thus by
Theorem~\ref{thm:vb_extends}, {\it the group $K_{-1} (R) \oplus
  NK_{0}(R) \oplus NK_{0} (R)$ is trivial if and only if every strict
  perfect complex of $\Rt \laurent x$-modules extends to~$\pp$ up to
  quasi-isomorphism}. For the truncated polynomial ring $R =
\bZ[T]/T^{m}$, where $m \geq 2$ is not a prime power, $K_{-1}(R) \neq
\{0\}$ \cite[Theorem~XII.10.6~(b)]{MR0249491}; it follows that
$p^{*}$~is not surjective in this case so that there exist strict
perfect complexes on~$T$ which do not extend to~$\pp$ up to
quasi-isomorphism.

\medbreak

The previous paragraph implies that we cannot simply ignore the
freeness assumption in Theorem~A as the proof breaks down at the very
first step. It thus comes as a surprise that the trivial observation
made at the beginning of the paper, that every strict perfect complex
can be complemented so that the resulting complex consists of free
modules, can be used to overcome this issue. As there is no (obvious)
way to control the homological behaviour of the complement near~$0$
and~$\infty$, a carefully re-phrased argument, as documented
in~\cite{MR3001337}, is needed to generalise Theorem~A to strict
perfect $\Rt\laurent x$\nbd-module complexes.

\bigbreak
\raggedright
\frenchspacing

\providecommand{\bysame}{\leavevmode\hbox to3em{\hrulefill}\thinspace}
\providecommand{\href}[2]{#2}


\begin{thebibliography}{Ran95}

\bibitem[Bas68]{MR0249491}
Hyman Bass, \emph{Algebraic {$K$}-theory}, W. A. Benjamin, Inc., New
  York-Amsterdam, 1968.

\bibitem[HQ13]{MR3001337}
Thomas H{\"u}ttemann and David Quinn, \emph{Finite domination and {N}ovikov
  rings. {I}terative approach}, Glasg. Math. J. \textbf{55} (2013), no.~1,
  pp.~145--160.

\bibitem[H{\"u}t11]{MR2978500}
Thomas H{\"u}ttemann, \emph{Double complexes and vanishing of {N}ovikov
  cohomology}, Serdica Math. J. \textbf{37} (2011), no.~4, pp.~295--304.

\bibitem[Qui73]{MR0338129}
Daniel Quillen, \emph{Higher algebraic {$K$}-theory. {I}}, Algebraic
  {$K$}-theory, {I}: {H}igher {$K$}-theories ({P}roc. {C}onf., {B}attelle
  {M}emorial {I}nst., {S}eattle, {W}ash., 1972), Lecture Notes in Math., vol.
  341, Springer, Berlin, 1973, pp.~85--147.

\bibitem[Ran85]{MR815431}
Andrew Ranicki, \emph{The algebraic theory of finiteness obstruction}, Math.
  Scand. \textbf{57} (1985), no.~1, pp.~105--126.

\bibitem[Ran95]{MR1341811}
\bysame, \emph{Finite domination and {N}ovikov rings}, Topology \textbf{34}
  (1995), no.~3, pp.~619--632.

\bibitem[TT90]{MR1106918}
R.~W. Thomason and Thomas Trobaugh, \emph{Higher algebraic {$K$}-theory of
  schemes and of derived categories}, The {G}rothendieck {F}estschrift, {V}ol.\
  {III}, Progr. Math., vol.~88, Birkh\"auser Boston, Boston, MA, 1990,
  pp.~247--435.

\end{thebibliography}
\end{document}